\documentclass[10pt]{article}
\usepackage{extsizes,graphicx,amssymb,amsmath,amsthm,amsfonts,mathtext, color,wrapfig,ytableau,mathtools}

\def\be{\begin{eqnarray}}
\def\ee{\end{eqnarray}}

\newtheorem{theorem}{Theorem}
\newtheorem{definition}[theorem]{Definition}
\newtheorem{lemma}[theorem]{Lemma}
\newtheorem{proposition}[theorem]{Proposition}
\newtheorem{conjecture}[theorem]{Conjecture}

\hoffset= - 0.8in         

\begin{document}

\thispagestyle{empty}
\baselineskip13.0pt

\ytableausetup{boxsize=5pt}

\vspace{-1ex}\hfill IITP/TH-25/26

\bigskip
\bigskip

\centerline{\Large{Elliptic Generalization of Cherednik-Macdonald-Mehta identities
}}

\vspace{3ex}

\centerline{\large{\emph{Sh.Shakirov\footnote{Institute for Information Transmission Problems, Moscow, Russia; shakirov.work@gmail.com }}}}

\vspace{3ex}

\centerline{ABSTRACT}

\bigskip

{\footnotesize
Integral identities for Macdonald polynomials play an important role in modern mathematics and mathematical physics. Especially interesting are the Cherednik-Macdonald-Mehta (CMM) identities, with profound connections to Double Affine Hecke Algebras (DAHA) and representation theory of quantum groups. These identities are central in refined Chern-Simons theory, where they lead to refined $S$ and $T$ matrices and ultimately to refined knot invariants. We suggest an elliptic generalization of CMM identities, where trigonometric Vandermonde products are replaced by theta functions. \linebreak At the same time Macdonald polynomials are promoted to Shiraishi functions -- distinguished elliptic functions with several interesting avatars, from the non-stationary Ruijsenaars problem in integrable systems, to equivariant $K$-theory characters of the affine Laumon space in algebraic geometry, to surface defect partition functions in 5d super Yang-Mills theory. \linebreak From the perspective of matrix models, we present an elliptic matrix model with a superintegrability property. We prove the suggested identities to the first order in the elliptic parameter.
}

\section{Introduction}

From the perspective of quantum many-body integrable systems, there are often 3 levels of complexity -- rational, trigonometric, and elliptic systems -- which can be further enriched to a $3 \times 3$ table, considering whether the dependence of Hamiltonians on coordinates is of one type and on momenta of the other \cite{M00, KS19}. Finding generalizations of certain properties to a higher level of this hierarchy is often a meaningful and interesting problem. Generalization from rational to trigonometric level is called $q$-deformation. In this paper, we deal with a generalization from trigonometric to elliptic level.

As the starting point of our discussion, we consider the Macdonald polynomials -- stationary wavefunctions of trig-trig Ruijsenaars-Schneider system that occupies the central square in the above mentioned table. For type $A_1$ this system can be simply defined by a pair of commuting Hamiltonian operators

\begin{align}
& \nonumber {\cal {\widehat H}}_1 f(x_1, x_2) = \dfrac{x_2 - t x_1}{x_2 - x_1} f(q x_1, x_2) + \dfrac{x_1 - t x_2}{x_1 - x_2} f(x_1, q x_2) \\
\nonumber \\
& {\cal {\widehat H}}_2 f(x_1, x_2) = f(q x_1, q x_2)
\end{align}
\smallskip\\
where $t,q$ are so far generic complex parameters. Polynomial wavefunctions $P_j(x_1,x_2)$ for integer $j \geq 0$ are solutions to ${\cal {\widehat H}}_1 P_j(x_1,x_2) = ( 1 + t q^j ) P_j(x_1,x_2)$ and ${\cal {\widehat H}}_2 P_j(x_1,x_2) = q^j P_j(x_1,x_2)$: Macdonald polynomials

\begin{align}
& \nonumber P_0(x_1,x_2) = 1, \\
& \nonumber P_1(x_1,x_2) = x_1 + x_2, \\
& \nonumber P_2(x_1,x_2) = x_1^2 + \dfrac{(1-t)(1+q)}{1-qt} x_1 x_2 + x_2^2, \\
& P_3(x_1,x_2) = x_1^3 + \dfrac{(1-t)(1+q+q^2)}{1-q^2t} (x_1^2 x_2 + x_1 x_2^2 ) + x_2^3, \ \ \ \ldots
\end{align}
\smallskip\\
In this paper we focus on type $A_1$ only, leaving generalization to higher rank for future research.

\pagebreak

Macdonald polynomials enjoy many interesting properties, among them the celebrated Cherednik-Macdonald-Mehta identities, which in their simplest instance for type $A_1$ can be presented as

\begin{align}
\nonumber & \dfrac{1}{2 \pi g} \int\limits_{-\infty}^{\infty} e^{-\frac{x_1^2}{2 g}} dx_1 \int\limits_{-\infty}^{\infty} e^{-\frac{x_2^2}{2 g}} dx_2 \ P_i\big( e^{x_1}, e^{x_2} \big) \ P_j\big( e^{x_1}, e^{x_2} \big) \prod\limits_{m = 0}^{\beta-1} (1 - q^m e^{x_1-x_2} )(1 - q^m e^{x_2-x_1} ) \\
& = q^{\frac{i(i+\beta) + ij + j(j+\beta)}{2}} P_i\big( t^{\frac{1}{2}} , t^{\frac{-1}{2}} \big) \ P_j\big( t^{\frac{1}{2}} q^{\frac{i}{2}} , t^{\frac{-1}{2}} q^{\frac{-i}{2}} \big) \cdot Z(g,\beta)
\end{align}
\smallskip\\
with $q = e^{g}, t = e^{\beta g}, g > 0$, natural $\beta = 1,2,3,\ldots$ and certain normalization constant $Z(g,\beta)$. It is easy to verify these identities for particular $i,j,\beta$, expanding the integrand as a polynomial in $e^{x_1}, e^{x_2}$ and making use of the main integration identity

\begin{align}
\dfrac{1}{\sqrt{2 \pi g}} \int\limits_{-\infty}^{\infty} e^{-\frac{x^2}{2 g}} e^{k x} dx = q^{k^2/2},  \ \ \ g > 0, \ \ \ k \in {\mathbb Z}
\label{ExpInt}
\end{align}
\smallskip\\
These identities were proven (in far greater generality) by I.Cherednik \cite{Ch97} and later by P. Etingof and A. Kirillov-Jr. \cite{EK97}. There are many useful applications: to mention a few, these identities play a central role in refined Chern-Simons theory \cite{AS11} where they provide a computation of the matrix element of the $TST$ operator in the basis of states corresponding to Macdonald polynomials; also these identities can be interpreted as superintegrability properties of unitary matrix models \cite{MMP24} because expectation value of a product of Macdonald polynomials is expressed through particular values of these same polynomials. 

In this paper, we concern ourselves with the following natural question: does there exist a nice generalization of CMM identities from trigonometric (difference) to elliptic level? Elliptic generalization is expected\footnote{Following intuition from studies in the AGT conjecture \cite{MMS25} where similar deformation appears.} to promote the trigonometric Vandermonde product to the elliptic one:

\begin{align}
\prod\limits_{m = 0}^{\beta-1} (1 - q^m e^{x_1-x_2} )(1 - q^m e^{x_2-x_1} ) \ \ \ \rightarrow \ \ \ \prod\limits_{m = 0}^{\beta-1} \theta_p(q^m e^{x_1-x_2} )\theta_p(q^m e^{x_2-x_1} )
\end{align}
\smallskip\\
where $\theta_p(x)$ is a theta function with elliptic parameter $p$ in multiplicative notation: 

\begin{align}
\theta_p(x) = \prod_{n = 0}^{\infty} (1 - p^n x)(1 - p^{n+1}/x)
\end{align}
\smallskip\\
To answer this question, one first needs to identify what should be the relevant elliptic deformation of Macdonald polynomials. Unfortunately, too simple for this purpose are eigenfunctions of the ell-trig (or simply called elliptic) Ruijsenaars-Schneider system,  

\begin{align}
& {\cal {\widehat H}}^{(ell)}_1 f(x_1, x_2) = \dfrac{\theta_p(t x_1/x_2)}{\theta_p(x_1/x_2)} f(q x_1, x_2) + \dfrac{\theta_p(t x_2/x_1)}{\theta_p(x_2/x_1)} f(x_1, q x_2) \\
\nonumber \\
& {\cal {\widehat H}}^{(ell)}_2 f(x_1, x_2) = f(q x_1, q x_2)
\end{align}
\smallskip\\
which have a form, for the first few cases,

\begin{align}
& P_0(x_1,x_2|p) = 1 + \dfrac{q (1 - t)(1 - t^2)}{t (1 - q t) (1 - q)} (x_1 x_2^{-1} + x_2 x_1^{-1}) p + O(p^2), \\
& P_1(x_1,x_2|p) = x_1 + x_2 + \dfrac{q (1 - t)(1 - q t^2)}{t (1 - q^2 t) (1 - q)} (x_1^2 x_2^{-1} + x_2^2 x_1^{-1}) p + O(p^2), \ldots
\end{align}
\smallskip\\
Such "candidate elliptic Macdonald functions" do not seem to provide non-trivial generalized CMM identities: at least, we were not able to find a meaningful generalization with the stationary case. 

From this perspective, one extra step is necessary to obtain the elliptic generalization of CMM identities: consider solutions to a non-stationary problem instead of stationary. Fortunately, the corresponding functions -- the Shiraishi functions (of type $A_1$) -- are present and sufficiently well studied in the literature \cite{Sh19, LNS20}. For type $A_1$ we denote them $P_j(x_1,x_2|p,s)$ and provide detailed definition and properties later in the paper. Compared to stationary wavefunctions of elliptic Ruijsenaars-Schneider system, they depend on one more parameter $s$, reducing to the stationary case at $s=1$:

\begin{align}
\underbrace{P_j(x_1,x_2)}_{\mbox{Macdonald polynomials}} \ \ \ \ \ \mathop{\longleftarrow}_{p=0} \ \ \ \ \  \underbrace{P_j(x_1,x_2|p)}_{\mbox{stationary elliptic Ruijsenaars}} \ \ \ \ \ \mathop{\longleftarrow}_{s=1} \ \ \ \ \ \underbrace{P_j(x_1,x_2|p,s)}_{\mbox{Shiraishi functions}}
\end{align}
\smallskip\\
To provide an example,

\begin{align}
& \nonumber P_0(x_1,x_2|p,s) = 1 + \dfrac{q (1 - t)(1 - s t^2)}{t (1 - q s t) (1 - q)} (x_1 x_2^{-1} + x_2 x_1^{-1}) p \\
& + \dfrac{(1 - t) (q - s t^2)}{t (1 - s t) (1 - q)} p + \dfrac{(1 - s t^2) (1 - q s) (1 - t) (q - s t)}{t(1 - q s t)(1 - st)(1-q)(1-s)} p + O(p^2),
\end{align}
\smallskip\\
and it is easy to check that, indeed, the limit $s \rightarrow 1$ of $P_0(x_1,x_2|p,s)$ normalized by $P_0(t^{\frac{1}{2}},t^{\frac{-1}{2}}|p,s)$ is the same as $P_0(x_1,x_2|p)$ normalized by $P_0(t^{\frac{1}{2}},t^{\frac{-1}{2}}|p)$. Shiraishi functions are intrinsically connected to algebraic geometry: they represent certain equivariant K-theory characters of affine Laumon space (see e.g. Theorem 3.3. and Proposition 3.4 in \cite{Sh19}). From a physics viewpoint, they are partition functions of surface defects in five-dimensional supersymmetric Yang-Mills theory, computed by localization \cite{BKK14}. These connections are behind the explicit formulas for Shiraishi functions as sums over partitions.
 
We suggest that Shiraishi functions satisfy the following elliptic generalization of CMM identities:

\begin{align}
\nonumber & \dfrac{1}{2 \pi g} \int\limits_{-\infty}^{\infty} e^{-\frac{x_1^2}{2 g}} dx_1 \int\limits_{-\infty}^{\infty} e^{-\frac{x_2^2}{2 g}} dx_2 \ P_i\big( e^{x_1}, e^{x_2} \big| p , p/s \big) \ P_j\big( e^{x_1}, e^{x_2} \big| p, s \big) \prod\limits_{m = 0}^{\beta-1} \theta_{p}(q^m e^{x_1-x_2} )\theta_{p}(q^m e^{x_2-x_1} ) \\
& = q^{\frac{i(i+\beta) + ij + j(j+\beta)}{2}} P_i\big( t^{\frac{1}{2}} , t^{\frac{-1}{2}} \big| 0, 0 \big) \ P_j\big( t^{\frac{1}{2}} q^{\frac{i}{2}} , t^{\frac{-1}{2}} q^{\frac{-i}{2}} \big| 0, 0\big) \ \dfrac{ P_j\big( t^{\frac{1}{2}} q^{\frac{i}{2}} , t^{\frac{-1}{2}} q^{\frac{-i}{2}} \big| p/s, s \big) }{ P_j( t^{\frac{1}{2}} q^{\frac{i}{2}} , t^{\frac{-1}{2}} q^{\frac{-i}{2}} \big| p/s, 0 \big) } \cdot Z(g,\beta,s,p) 
\label{EllipticCMM}
\end{align}
\smallskip\\
with $q = e^{g}, t = e^{\beta g}$, real $g > 0$, natural $\beta = 1,2,3,\ldots$ and certain normalization constant $Z(g,\beta,s,p)$. Just like for the ordinary CMM identities, it is easy to verify these identities for particular $i,j,\beta$, using the explicit formulas for $P_i(x_1,x_2|p,s)$ to desired order in $p$ as a Laurent polynomial in $e^{x_1}, e^{x_2}$ and making use of (\ref{ExpInt}). It is also easy to see why the stationary wavefunctions are not sufficient: indeed, specialization $s=1$ is not compatible with some of the necessary transformations of $(p,s)$ parameters. In this paper, we prove eq. (\ref{EllipticCMM}) for arbitrary $i,j$ to the first order in the series expansion in $p$, see Theorem 9.

We are hopeful that this could attract more attention to Shiraishi functions as a natural coordinate-elliptic generalization of Macdonald functions, and may lead to further research in elliptic refined Chern-Simons theory \cite{AS23}. Generalizations are expected at least to type $A_n$, and non-integer $\beta$. A separate interesting story is the potential relations to the formulas we suggest in this paper to representation theory of DIM algebras.

\section{Macdonald polynomials and Shiraishi functions}

Let us present side by side definitions for Macdonald polynomials and for their elliptic generalizations -- the Shiraishi functions. For the Macdonald polynomials, it is well-known that for type $A_1$ they are the same as continuous $q$-ultraspherical polynomials (see e.g. \cite{M03}, eq. (6.3.7)). For Shiraishi functions, the definition is a sum over partitions ( see e.g. \cite{Sh19}, Definition 1.2 specialized to $N=2$).

\paragraph{}We fix the ground field ${\mathbf{k}} = {\mathbb C}(q,t,s)$ to be the field of rational functions in $q,t,s$.

\begin{definition}
Macdonald polynomials of type $A_1$ with index $j = 0,1,2,\ldots$ are given by

\begin{align}
P_j(x_1,x_2) = x_1^j \sum\limits_{n = 0}^{j} \ \left(\frac{x_2}{x_1}\right)^n \ \prod\limits_{i = 0}^{n-1} \dfrac{(1 - q^{j-i})(1 - tq^{i})}{(1 - t q^{j-i-1})(1 - q^{i+1})} \ \end{align}
\end{definition}

\begin{definition}
Shiraishi functions of type $A_1$ with index $j = 0,1,2,\ldots$ are formal power series

\begin{align}
P_j \big( x_1, x_2, \big| p, s \big) = x_1^j \ \sum\limits_{\lambda, \mu} \ \dfrac{{\cal N}_{\lambda, \lambda}^{(0)}\left( \frac{q}{t} \right){\cal N}_{\lambda, \mu}^{(1)}\left( \frac{q y}{t} \right){\cal N}_{\mu, \lambda}^{(-1)}\left( \frac{q}{t y} \right){\cal N}_{\mu, \mu}^{(0)}\left( \frac{q}{t}\right)}{{\cal N}_{\lambda, \lambda}^{(0)}\left( 1\right){\cal N}_{\lambda, \mu}^{(1)}\left( y\right){\cal N}_{\mu, \lambda}^{(-1)}\left( y^{-1} \right){\cal N}_{\mu, \mu}^{(0)}\left( 1\right)} \ \prod\limits_{a = 1}^{\ell(\lambda)} \left( \frac{t X_{a}}{q} \right)^{\lambda_a}
\prod\limits_{a = 1}^{\ell(\mu)} \left( \frac{t X_{a+1}}{q} \right)^{\mu_a}
\label{ShiraishiFunction}
\end{align}
\smallskip\\
where the sum is over all pairs of partitions $\lambda = (\lambda_1 \geq \lambda_2 \geq \ldots \geq 0)$ and $\mu = (\mu_1 \geq \mu_2 \geq \ldots \geq 0)$, the length (number of nonzero components) of partition $P$ is denoted $\ell(P)$, and

\begin{align}
& \nonumber {\cal N}_{P,Q}^{(k)}\big( u \big) = \mathop{ \prod\limits_{b = 1}^{\ell(P)} \prod\limits_{a = 1}^{b} }_{b - a - k \ \rm{mod} \ 2 = 0} \prod\limits_{m = 0}^{P_b-P_{b+1}-1} \big( 1 - u q^m q^{-Q_a + P_{b+1}} s^{\frac{-a+b}{2}} \big) \\ 
& \times \mathop{ \prod\limits_{b = 1}^{\ell(Q)} \prod\limits_{a = 1}^{b} }_{b - a + k + 1 \ \rm{mod} \ 2 = 0} \prod\limits_{m = 0}^{Q_b-Q_{b+1}-1} \big( 1 - u q^m q^{P_a - Q_{b}} s^{\frac{a-b-1}{2}} \big)
\label{orbNek}
\end{align}
\smallskip\\
are Nekrasov factors. The indices and variables are parametrized as follows:

\begin{align}
y = s^{-1/2} t^{-1} q^{-j}, \ \ \ \ \ X_a = \left\{ \begin{array}{ccc} 
p x_1/x_2, \ \ \ a \mbox{ even } \\
\\
x_2/x_1, \ \ \ a \mbox{ odd }
\end{array} \right.
\end{align}
\end{definition}

We briefly comment that Shiraishi functions have an important interpretation in algebraic geometry -- as certain equivariant $K$-theory characters of the affine Laumon space (see Proposition 3.4. in \cite{Sh19}). At the same time they have an interpretation in supersymmetric gauge theory -- as partition functions in 5d super Yang-Mills theory with adjoint matter in the presence of certain surface defect (see eqs. (4.11) - (4.14) in \cite{BKK14} which somewhat implicitly define the same function as eq. (\ref{ShiraishiFunction})).

\begin{proposition} 
Function $P_j \big( x_1, x_2, \big| p, s \big)$ is a formal power series in $p$ which is a Laurent polynomial in $x_1,x_2$ to every given order in $p$, with coefficients in ${\mathbf{k}}$.
\end{proposition}
\begin{proof}
First, a pair of partitions $\lambda, \mu$ contributes to the sum (\ref{ShiraishiFunction}) a value of degree $d = \sum_{i \geq 1} \lambda_{2i} + \sum_{i \geq 1} \mu_{2i - 1}$ in $p$. Second, the summand in (\ref{ShiraishiFunction}) is proportional to 

\begin{align}
{\cal N}_{\mu, \lambda}^{(-1)}\left( \frac{q}{t y} \right)
\end{align}
\smallskip\\
that vanishes whenever for at least one value of $a = 1, \ldots, \ell(\lambda)$ we have

\begin{align}
\mu_a \in [  \lambda_{a+1} - j, \ldots, \lambda_a - j - 1 ]
\end{align}
\smallskip\\
due to vanishing of factors at $b = a$ in the second line of (\ref{orbNek}). This implies that any pair $\lambda, \mu$ that gives a non-vanishing contribution must have $\lambda_1 \leq \mu_1 + j$: otherwise, if $\mu_1 < \lambda_1 - j$, the conditions force also $\mu_1 < \lambda_2 - j$, hence $\mu_2 < \lambda_2 - j$, hence $\mu_2 < \lambda_3 - j$, and so on, until finally we would arrive at conclusion that $\mu_{\ell(\lambda)} < - j$ which is impossible, a contradiction. Since there are only finitely many pairs of partitions $\lambda,\mu$ that satisfy sumultaneously $\lambda_1 \leq \mu_1 + j$ and $\sum_{i \geq 1} \lambda_{2i} + \sum_{i \geq 1} \mu_{2i - 1} \leq d_{\mbox{max}}$ for a given order $d_{\mbox{max}}$ in $p$, we conclude that $P_j \big( x_1, x_2, \big| p, s \big)$ is a sum of finitely many monomials in $x_1^{\pm 1},x_2^{\pm 1}$ up to every given order in $p$. It is therefore a finite Laurent polynomial in $x_1,x_2$ up to every given order in $p$, and by construction the coefficients are rational functions of $q,t,s$.
\end{proof}

\begin{proposition}
In the limit $p = 0$, Shiraishi function reduces to Macdonald polynomial:

\begin{align}
P_j \big( x_1, x_2, \big| 0, s \big) = P_j \big( x_1, x_2 \big)
\end{align}
\end{proposition}
\begin{proof}At $p=0$ only finitely many pairs of partitions give a non-zero contribution to the sum (\ref{ShiraishiFunction}): namely, $\lambda = (n), \mu = \varnothing$ with $n = 0,1,2,\ldots,j$. Direct computation shows that this contribution is

\begin{align}
\left(\frac{x_2}{x_1}\right)^n \ \prod\limits_{i = 0}^{n-1} \dfrac{(1 - q^{j-i})(1 - tq^{i})}{(1 - t q^{j-i-1})(1 - q^{i+1})}
\end{align}
\smallskip\\
which is nothing but the summand in Definition 1. 
\end{proof}

\begin{proposition}
To the first order in the elliptic parameter $p$, Shiraishi function has the following expansion in Macdonald polynomials:

\begin{align}
& \nonumber P_j \big( x_1, x_2 \big| p, s \big) = P_j \big( x_1, x_2 \big) + \\
\nonumber \\
& + \left( c_0(j) (x_1 x_2)^{-1} P_{j+2} \big( x_1, x_2 \big) + c_1(j) P_{j} \big( x_1, x_2 \big) + c_2(j) x_1 x_2 P_{j-2} \big( x_1, x_2 \big) \right) p + O(p^2) 
\end{align}
\smallskip\\
with
\begin{align}
c_0(j) = \dfrac{q (1 - t)(1 - q^j s t^2)}{t (1-q) (1 - q^{j+1} s t)}
\end{align}
\begin{align}
c_1(j) = \dfrac{(q - s t) (1 - t) \big( 2 q^{2j} t - q^{j-1} (1 - q)^2 - q^{j} (q + 1)(t + 1/t) + 2 t^{-1} \big) }{(1 - s) (1 - q^{j+1} t) (1-q^{j-1} t) (1-q)}
\end{align}
\begin{align}
c_2(j) = \dfrac{ (1-t) (1 - q^{j-2} t^2) (1 - q^j) (1 - q^{j-1}) (s - q^j) (1 - q^{j-1} t^2)}{(1-q)(1 - q^{j-1} t)^2 (1 - q^{j-2} t) (s - q^{j-1} t) (1 - q^j t)}
\end{align}
\end{proposition}
\begin{proof}We already proved this to the zeroth order in the elliptic parameter. Counting the degrees of $X_2$, one can see that the following four types of pairs give contributions proportional to $p^1$:

\begin{align}
\lambda = (n), \ \ \ \mu = (1,1), \ \ \ \ \ n = 0, \ldots, j + 1
\end{align}
\begin{align}
\lambda = (n), \ \ \ \mu = (1), \ \ \ \ \ n = 0, \ldots, j + 1
\end{align}
\begin{align}
\lambda = (n,1), \ \ \ \mu = \varnothing, \ \ \ \ \ n = 1, \ldots, j
\end{align}
\begin{align}
\lambda = (n,1,1), \ \ \ \mu = \varnothing, \ \ \ \ \ n = 1, \ldots, j
\end{align}
\smallskip\\
For each $j$ there are finitely many of those. Direct computation gives contribution of the first type

\begin{align}
F^{(1)}_n = \left(\frac{x_2}{x_1}\right)^n \dfrac{(1-t) (q-s t) (1-q^j s t^2)(1-q^{j+1-n} s)}{t(1-q) (1-s)(1-q^{j+1} s t)(1-q^{j-n} s t)} \prod\limits_{i = 0}^{n-1} \dfrac{(1-q^{j-i+1}) (1-q^{i} t)}{(1-q^{j-i} t) (1-q^{i+1})}
\end{align}
\smallskip\\
contribution of the second type

\begin{align}
F^{(2)}_n = \left(\frac{x_2}{x_1}\right)^{n-1} \dfrac{q(1-t)(1-q^{j-n}ts^2)}{t(1-q)(1-q^{j+1-n}st)} \prod\limits_{i = 0}^{n-1} \dfrac{(1-q^{j-i+1}) (1-q^{i} t)}{(1-q^{j-i} t) (1-q^{i+1})}
\end{align}
\smallskip\\
contribution of the third type

\begin{align}
F^{(3)}_n = \left(\frac{x_2}{x_1}\right)^{n-1} \dfrac{(1-t) (1-q^j) (s t-q^n) (1-q^{j-1} t^2)}{t (1-q)(1-q^j t) (1-q^{j-1} t) (s-q^{n-1})} \prod\limits_{i = 0}^{n-2} \dfrac{(1-q^{j-i-1}) (1-q^{i} t)}{(1-q^{j-2-i} t) (1-q^{i+1})}
\end{align}
\smallskip\\
and finally contribution of the fourth type

\begin{align}
F^{(4)}_n = \left(\frac{x_2}{x_1}\right)^{n} \dfrac{(q-s t)(1-t)(1-q^j)(s-q^j)(1-q^{j-1} t^2)(s-q^{n-1} t)}{t(1-s)(1-q)(s-q^n)(1-q^j t) (s-q^{j-1} t) (1-q^{j-1}t)}
 \prod\limits_{i = 0}^{n-2} \dfrac{(1-q^{j-i-1}) (1-q^{i} t)}{(1-q^{j-2-i} t) (1-q^{i+1})}
\end{align}
\smallskip\\
We therefore need to prove that

\begin{align}
& \nonumber c_0(j) (x_1 x_2)^{-1} P_{j+2} \big( x_1, x_2 \big) + c_1(j) P_{j} \big( x_1, x_2 \big) + c_2(j) x_1 x_2 P_{j-2} \big( x_1, x_2 \big) 
 = \\
& x_1^{j} \sum\limits_{n = 0}^{j+1} \big( F^{(1)}_n + F^{(2)}_n \big) + x_1^{j} \sum\limits_{n = 1}^{j} \big( F^{(3)}_n + F^{(4)}_n \big)
\end{align}
\smallskip\\
Dividing by $x_1^j$ and extracting term in front of $(x_2/x_1)^n$, we have to prove that

\begin{align}
& \nonumber c_0(j) \prod\limits_{i = 0}^{n} \dfrac{(1 - q^{j+2-i})(1 - tq^{i})}{(1 - t q^{j-i+1})(1 - q^{i+1})} + c_1(j) \delta_{n > -1} \prod\limits_{i = 0}^{n-1} \dfrac{(1 - q^{j-i})(1 - tq^{i})}{(1 - t q^{j-i-1})(1 - q^{i+1})} + c_2(j) \delta_{n > 0} \prod\limits_{i = 0}^{n-2} \dfrac{(1 - q^{j-2-i})(1 - tq^{i})}{(1 - t q^{j-i-3})(1 - q^{i+1})} 
 = \\
& \nonumber \delta_{n > -1} \dfrac{(1-t) (q-s t) (1-q^j s t^2)(1-q^{j+1-n} s)}{t(1-q) (1-s)(1-q^{j+1} s t)(1-q^{j-n} s t)} \prod\limits_{i = 0}^{n-1} \dfrac{(1-q^{j-i+1}) (1-q^{i} t)}{(1-q^{j-i} t) (1-q^{i+1})} + \\
& \nonumber \dfrac{q(1-t)(1-q^{j-n-1}ts^2)}{t(1-q)(1-q^{j-n}st)} \prod\limits_{i = 0}^{n} \dfrac{(1-q^{j-i+1}) (1-q^{i} t)}{(1-q^{j-i} t) (1-q^{i+1})} + \\
& \nonumber \delta_{n > -1} \dfrac{(1-t) (1-q^j) (s t-q^{n+1}) (1-q^{j-1} t^2)}{t (1-q)(1-q^j t) (1-q^{j-1} t) (s-q^{n})} \prod\limits_{i = 0}^{n-1} \dfrac{(1-q^{j-i-1}) (1-q^{i} t)}{(1-q^{j-2-i} t) (1-q^{i+1})} + \\
& \delta_{n > 0} \dfrac{(q-s t)(1-t)(1-q^j)(s-q^j)(1-q^{j-1} t^2)(s-q^{n-1} t)}{t(1-s)(1-q)(s-q^n)(1-q^j t) (s-q^{j-1} t) (1-q^{j-1}t)}
 \prod\limits_{i = 0}^{n-2} \dfrac{(1-q^{j-i-1}) (1-q^{i} t)}{(1-q^{j-2-i} t) (1-q^{i+1})}
\end{align}
\smallskip\\
for every $n \geq -1$. For $n = -1$, this amounts to equality between the first term on the l.h.s and second term on the r.h.s, which are manifestly equal due to the formula for $c_0(j)$. For $n = 0$, this amounts to

\begin{align}
& \nonumber c_0(j) \dfrac{(1 - q^{j+2})(1 - t)}{(1 - t q^{j+1})(1 - q)} + c_1(j)
 = \\
& \nonumber \dfrac{(1-t) (q-s t) (1-q^j s t^2)(1-q^{j+1} s)}{t(1-q) (1-s)(1-q^{j+1} s t)(1-q^{j} s t)} + \\
& \nonumber \dfrac{q(1-t)(1-q^{j-n-1}ts^2)}{t(1-q)(1-q^{j-n}st)} \dfrac{(1-q^{j+1}) (1-t)}{(1-q^{j} t) (1-q)} + \\
& \dfrac{(1-t) (1-q^j) (s t-q) (1-q^{j-1} t^2)}{t (1-q)(1-q^j t) (1-q^{j-1} t) (s-1)}
\end{align}
\smallskip\\
which is a valid identity of rational functions in $q^j$ over ${\mathbf k}$. For $n \geq 1$, we can divide by 
\begin{align}
\prod\limits_{i = 0}^{n-2} \dfrac{(1-q^{j-i-1}) (1-q^{i} t)}{(1-q^{j-2-i} t) (1-q^{i+1})}
\end{align}
to obtain

\begin{align}
& \nonumber c_0(j) \dfrac{(1-q^j) (1-q^{j+1})(1-q^{j+2}) (1-q^{n-1}t) (1-q^n t)}{(1-q^n)(1-q^{n+1}) (1-q^{n+1}) (1-q^j t) (1-q^{j-1}t)(1-q^{j+1}t)} + c_1(j) \dfrac{(1-q^j) (1-q^{n-1} t)}{(1-q^n) (1-q^{j-1}t)} \\
& \nonumber + c_2(j) \dfrac{(1-q^{j-2} t) (1-q^{j-n})}{(1-q^{j-1}) (1-q^{j-n-1} t)} 
 = \\
& \nonumber \dfrac{(1-t) (q-s t) (1-q^j s t^2)(1-q^{j+1-n} s)}{t(1-q) (1-s)(1-q^{j+1} s t)(1-q^{j-n} s t)} \dfrac{(1-q^j) (1-q^{j+1}) (1-q^{n-1}t) (1-q^{j-n} t)}{(1-q^n) (1-q^j t) (1-q^{j-1}t)(1-q^{j+1-n})} + \\
& \nonumber \dfrac{q(1-t)(1-q^{j-n-1}ts^2)}{t(1-q)(1-q^{j-n}st)} \dfrac{(1-q^j) (1-q^{j+1}) (1-q^{n-1}t) (1-q^n t)}{(1-q^n) (1-q^{n+1}) (1-q^j t) (1-q^{j-1}t)} + \\
& \nonumber \dfrac{(1-t) (1-q^j) (s t-q^{n+1}) (1-q^{j-1} t^2)}{t (1-q)(1-q^j t) (1-q^{j-1} t) (s-q^{n})} \dfrac{(1-q^{n-1} t) (1-q^{j-n})}{(1-q^n) (1-q^{j-n-1} t)} + \\
& \dfrac{(q-s t)(1-t)(1-q^j)(s-q^j)(1-q^{j-1} t^2)(s-q^{n-1} t)}{t(1-s)(1-q)(s-q^n)(1-q^j t) (s-q^{j-1} t) (1-q^{j-1}t)}
\end{align}
which is a valid identity of rational functions in $q^j$ over ${\mathbf k}$.\end{proof}

\section{Elliptic Cherednik-Macdonald-Mehta identities}

We conjecture that Shiraishi functions satisfy the following elliptic generalization of CMM identities:

\begin{conjecture}
Fix real $g > 0$, integer $\beta \geq 1$, complex $s$, and let $q = e^{g}, t = e^{\beta g}$. For any $i,j \geq 0$, the following is a valid identity of formal power series in $p$, to arbitrary order $O(p^n)$:

\begin{align}
\nonumber & \dfrac{1}{2 \pi g} \int\limits_{-\infty}^{\infty} e^{-\frac{x_1^2}{2 g}} dx_1 \int\limits_{-\infty}^{\infty} e^{-\frac{x_2^2}{2 g}} dx_2 \ P_i\big( e^{x_1}, e^{x_2} \big| p , p/s \big) \ P_j\big( e^{x_1}, e^{x_2} \big| p, s \big) \prod\limits_{m = 0}^{\beta-1} \theta_{p}(q^m e^{x_1-x_2} )\theta_{p}(q^m e^{x_2-x_1} ) \\
& = q^{\frac{i(i+\beta) + ij + j(j+\beta)}{2}} P_i\big( t^{\frac{1}{2}} , t^{\frac{-1}{2}} \big| 0, 0 \big) \ P_j\big( t^{\frac{1}{2}} q^{\frac{i}{2}} , t^{\frac{-1}{2}} q^{\frac{-i}{2}} \big| 0, 0\big) \ \dfrac{ P_j\big( t^{\frac{1}{2}} q^{\frac{i}{2}} , t^{\frac{-1}{2}} q^{\frac{-i}{2}} \big| p/s, s \big) }{ P_j( t^{\frac{1}{2}} q^{\frac{i}{2}} , t^{\frac{-1}{2}} q^{\frac{-i}{2}} \big| p/s, 0 \big) } \cdot Z(g,\beta,s,p) 
\label{EllipticCMMconj}
\end{align}
\smallskip\\
We verified this conjecture for $1 \leq \beta \leq 4$ and $0 \leq i,j \leq 3$ up to order $O(p^3)$.

\end{conjecture}

In this paper, we will prove Conjecture 6 to order $O(p^{n = 1})$. For this, we will need the following

\begin{lemma}
The elliptic Vandermonde product admits the following series expansion:

\begin{align}
\dfrac{\prod\limits_{m = 0}^{\beta-1} \theta_{p}(q^m e^{x_1-x_2} )\theta_{p}(q^m e^{x_2-x_1} )}{\prod\limits_{m = 0}^{\beta-1} (1 - q^m e^{x_1-x_2} )(1 - q^m e^{x_2-x_1} )} = 1 - \dfrac{(1-t)(t+q)}{t(1-q)} \left(e^{x_1 - x_2} + e^{x_2 - x_1}\right) p + O(p^2)
\end{align}
\end{lemma}
\begin{proof}
Up to order $O(p^1)$, the l.h.s. is the same as the product

\begin{align}
\prod\limits_{m = 0}^{\beta-1} (1 - p q^m e^{x_1-x_2} )(1 - p q^{-m} e^{x_2-x_1})
\end{align}
\smallskip\\
and the same as a sum

\begin{align}
1 - p \sum\limits_{m = 0}^{\beta-1} \left( q^m e^{x_1-x_2} + q^{-m} e^{x_2-x_1} \right)
\end{align}
\smallskip\\
which is a geometric sum straightforwardly computed to yield the stated result.

\end{proof}

\begin{lemma}
To the first order in the elliptic parameter $p$, the Shiraishi function under the integral with transformed parameters has the following expansion in Macdonald polynomials:

\begin{align}
& \nonumber P_j \big( x_1, x_2 \big| p, p/s \big) = P_j \big( x_1, x_2 \big) + \\
\nonumber \\
& + \left( u_0(j) (x_1 x_2)^{-1} P_{j+2} \big( x_1, x_2 \big) + u_1(j) P_{j} \big( x_1, x_2 \big) + u_2(j) x_1 x_2 P_{j-2} \big( x_1, x_2 \big) \right) p + O(p^2) 
\end{align}
\smallskip\\
with
\begin{align}
u_0(j) = \dfrac{q(1-t)}{t (1-q)}
\end{align}
\begin{align}
u_1(j) = \dfrac{q (1 - t) \big( 2 q^{2j} t - q^{j-1} (1 - q)^2 - q^{j} (q + 1)(t + 1/t) + 2 t^{-1} \big) }{t(1 - q^{j+1} t) (1-q^{j-1} t) (1-q)}
\end{align}
\begin{align}
u_2(j) = \dfrac{ q (1 - t) (1 - q^j) (1 - q^{j - 1}) (1 - q^{j - 2} t^2) (1 - q^{j - 1} t^2)}{t (1 - q) (1 - q^j t) (1 - q^{j-1} t)^2 (1 - q^{j-2} t)}
\end{align}
\smallskip\\
\end{lemma}
\begin{proof}
This is a corollary of Prop. 5. Replacing $s \rightarrow p/s$ and expanding in $p$, we obtain $u$ from $c$.

\end{proof}

\paragraph{}We are now ready to present our main result -- the proof of Conjecture 6 up to order $O(p^1)$.

\begin{theorem}
Conjecture 6 holds true to order $O(p^1)$.
\end{theorem}
\begin{proof}
We start by expanding the integrand into series in $p$ up to $O(p^1)$:
\begin{align}
& \nonumber \dfrac{ P_i\big( e^{x_1}, e^{x_2} \big| p , p/s \big) \ P_j\big( e^{x_1}, e^{x_2} \big| p, s \big) \prod\limits_{m = 0}^{\beta-1} \theta_{p}(q^m e^{x_1-x_2} )\theta_{p}(q^m e^{x_2-x_1} ) }{\prod\limits_{m = 0}^{\beta-1} (1 - q^m e^{x_1-x_2} )(1 - q^m e^{x_2-x_1} )} = P_i\big( e^{x_1}, e^{x_2} \big) \ P_j\big( e^{x_1}, e^{x_2} \big)  + \\
& \nonumber + \left( u_0(i) e^{-x_1 - x_2} P_{i+2} \big( e^{x_1}, e^{x_2} \big) + u_1(i) P_{i} \big( e^{x_1}, e^{x_2} \big) + u_2(i) e^{x_1 + x_2} P_{i-2} \big( e^{x_1}, e^{x_2} \big) \right) \ P_j\big( e^{x_1}, e^{x_2} \big) p  + \\
& \nonumber + \left( c_0(j) e^{-x_1 - x_2} P_{j+2} \big( e^{x_1}, e^{x_2} \big) + c_1(j) P_{j} \big( e^{x_1}, e^{x_2} \big) + c_2(j) e^{x_1 + x_2} P_{j-2} \big( e^{x_1}, e^{x_2} \big) \right) \ P_i\big( e^{x_1}, e^{x_2} \big) p + \\
& \left( - \dfrac{(1-t)(t+q)}{t(1-q)} \left(e^{x_1 - x_2} + e^{x_2 - x_1}\right) \right) P_i\big( e^{x_1}, e^{x_2} \big) \ P_j\big( e^{x_1}, e^{x_2} \big) p + O(p^2)
\end{align}
\smallskip\\
Using the Pieri rule for Macdonald polynomials, we compute

\begin{align}
& \nonumber \left( - \dfrac{(1-t)(t+q)}{t(1-q)} \left(e^{x_1 - x_2} + e^{x_2 - x_1}\right) \right) P_i\big( e^{x_1}, e^{x_2} \big) = \\
& v_0(i) e^{-x_1 - x_2} P_{i+2} \big( e^{x_1}, e^{x_2} \big) + v_1(i) P_{i} \big( e^{x_1}, e^{x_2} \big) + v_2(i) e^{x_1 + x_2} P_{i-2} \big( e^{x_1}, e^{x_2} \big) 
\end{align}
\smallskip\\
with
\begin{align}
v_0(j) = - \dfrac{(1-t)(t+q)}{t(1-q)} 
\end{align}
\begin{align}
v_1(j) = -\dfrac{q^{j-1}(1-t)^2(1+q)(t^2-q^2)}{t(1-q)(1-tq^{j-1})(1-tq^{j+1})}
\end{align}
\begin{align}
v_2(j) = -\dfrac{(1-t)(t+q)(1-q^j) (1-q^{j-1}) (1+t q^{j-2}) (1-t^2 q^{j-1}) (1-t^2 q^{j-2})}{t(1-q)(1-t^2 q^{2 j-4}) (1-t q^j) (1-t q^{j-1})^2}
\end{align}
\smallskip\\
Therefore, we obtain

\begin{align}
& \nonumber \dfrac{ P_i\big( e^{x_1}, e^{x_2} \big| p , p/s \big) \ P_j\big( e^{x_1}, e^{x_2} \big| p, s \big) \prod\limits_{m = 0}^{\beta-1} \theta_{p}(q^m e^{x_1-x_2} )\theta_{p}(q^m e^{x_2-x_1} ) }{\prod\limits_{m = 0}^{\beta-1} (1 - q^m e^{x_1-x_2} )(1 - q^m e^{x_2-x_1} )} = P_i\big( e^{x_1}, e^{x_2} \big) \ P_j\big( e^{x_1}, e^{x_2} \big)  + \\
& \nonumber + \left( u_0(i) e^{-x_1 - x_2} P_{i+2} \big( e^{x_1}, e^{x_2} \big) + u_1(i) P_{i} \big( e^{x_1}, e^{x_2} \big) + u_2(i) e^{x_1 + x_2} P_{i-2} \big( e^{x_1}, e^{x_2} \big) \right) \ P_j\big( e^{x_1}, e^{x_2} \big) p  + \\
& \nonumber + \left( c_0(j) e^{-x_1 - x_2} P_{j+2} \big( e^{x_1}, e^{x_2} \big) + c_1(j) P_{j} \big( e^{x_1}, e^{x_2} \big) + c_2(j) e^{x_1 + x_2} P_{j-2} \big( e^{x_1}, e^{x_2} \big) \right) \ P_i\big( e^{x_1}, e^{x_2} \big) p + \\
& \emph{} \hspace{-4ex} \left( v_0(i) e^{-x_1 - x_2} P_{i+2} \big( e^{x_1}, e^{x_2} \big) + v_1(i) P_{i} \big( e^{x_1}, e^{x_2} \big) + v_2(i) e^{x_1 + x_2} P_{i-2} \big( e^{x_1}, e^{x_2} \big) \right) \ P_j\big( e^{x_1}, e^{x_2} \big) p + O(p^2)
\label{SeriesIntegrand}
\end{align}
\smallskip\\
Using the ordinary CMM identities for Macdonald polynomials

\begin{align}
\nonumber & \dfrac{1}{2 \pi g} \int\limits_{-\infty}^{\infty} e^{-\frac{x_1^2}{2 g}} dx_1 \int\limits_{-\infty}^{\infty} e^{-\frac{x_2^2}{2 g}} dx_2 \ P_i\big( e^{x_1}, e^{x_2} \big) \ P_j\big( e^{x_1}, e^{x_2} \big) \prod\limits_{m = 0}^{\beta-1} (1 - q^m e^{x_1-x_2} )(1 - q^m e^{x_2-x_1} ) \\
& = q^{\frac{i(i+\beta) + ij + j(j+\beta)}{2}} P_i\big( t^{\frac{1}{2}} , t^{\frac{-1}{2}} \big) \ P_j\big( t^{\frac{1}{2}} q^{\frac{i}{2}} , t^{\frac{-1}{2}} q^{\frac{-i}{2}} \big) \cdot Z(g,\beta)
\end{align}
\begin{align}
\nonumber & \dfrac{1}{2 \pi g} \int\limits_{-\infty}^{\infty} e^{-\frac{x_1^2}{2 g}} dx_1 \int\limits_{-\infty}^{\infty} e^{-\frac{x_2^2}{2 g}} dx_2 \ e^{x_1+x_2} P_i\big( e^{x_1}, e^{x_2} \big) \ P_j\big( e^{x_1}, e^{x_2} \big) \prod\limits_{m = 0}^{\beta-1} (1 - q^m e^{x_1-x_2} )(1 - q^m e^{x_2-x_1} ) \\
& = q^{\frac{i(i+\beta) + ij + j(j+\beta)}{2} + i + j + 1} P_i\big( t^{\frac{1}{2}} , t^{\frac{-1}{2}} \big) \ P_j\big( t^{\frac{1}{2}} q^{\frac{i}{2}} , t^{\frac{-1}{2}} q^{\frac{-i}{2}} \big) \cdot Z(g,\beta)
\end{align}
\begin{align}
\nonumber & \dfrac{1}{2 \pi g} \int\limits_{-\infty}^{\infty} e^{-\frac{x_1^2}{2 g}} dx_1 \int\limits_{-\infty}^{\infty} e^{-\frac{x_2^2}{2 g}} dx_2 \ e^{-x_1-x_2} P_i\big( e^{x_1}, e^{x_2} \big) \ P_j\big( e^{x_1}, e^{x_2} \big) \prod\limits_{m = 0}^{\beta-1} (1 - q^m e^{x_1-x_2} )(1 - q^m e^{x_2-x_1} ) \\
& = q^{\frac{i(i+\beta) + ij + j(j+\beta)}{2} - i - j + 1} P_i\big( t^{\frac{1}{2}} , t^{\frac{-1}{2}} \big) \ P_j\big( t^{\frac{1}{2}} q^{\frac{i}{2}} , t^{\frac{-1}{2}} q^{\frac{-i}{2}} \big) \cdot Z(g,\beta)
\end{align}
\smallskip\\
we can integrate (\ref{SeriesIntegrand}) to obtain the l.h.s. of the elliptic CMM identity, in the form

{\fontsize{8pt}{0pt}{\begin{align}
\nonumber & \mbox{lhs} = \dfrac{q^{\frac{-i(i+\beta) - ij - j(j+\beta)}{2}} }{2 \pi g Z(g,\beta)} \int\limits_{-\infty}^{\infty} e^{-\frac{x_1^2}{2 g}} dx_1 \int\limits_{-\infty}^{\infty} e^{-\frac{x_2^2}{2 g}} dx_2 \ P_i\big( e^{x_1}, e^{x_2} \big| p , p/s \big) \ P_j\big( e^{x_1}, e^{x_2} \big| p, s \big) \prod\limits_{m = 0}^{\beta-1} \theta_{p}(q^m e^{x_1-x_2} )\theta_{p}(q^m e^{x_2-x_1} ) \\
& \nonumber = S_{i,j} + ( u_0(i) t q^{i+1} S_{i+2,j} + u_1(i) S_{i,j} + u_2(i) q^{-i+1} t^{-1} S_{i-2,j} ) p + ( c_0(j) t q^{j+1} S_{i,j+2} + c_1(j) S_{i,j} + c_2(j) q^{-j+1} t^{-1} S_{i,j-2} ) p + \\
& ( v_0(i) t q^{i+1} S_{i+2,j} + v_1(i) S_{i,j} + v_2(i) q^{-i+1} t^{-1} S_{i-2,j} ) p + O(p^2) 
\label{EquationLHS}
\end{align}}}
\smallskip\\
where we denoted

\begin{align}
S_{i,j} = P_i\big( t^{\frac{1}{2}} , t^{\frac{-1}{2}} \big) \ P_j\big( t^{\frac{1}{2}} q^{\frac{i}{2}} , t^{\frac{-1}{2}} q^{\frac{-i}{2}} \big)
\end{align}
\smallskip\\
What we need to prove is that expression (\ref{EquationLHS}) is equal to the r.h.s. of the elliptic CMM identity,

\begin{align}
\nonumber & \mbox{rhs} = P_i\big( t^{\frac{1}{2}} , t^{\frac{-1}{2}} \big| 0, 0 \big) \ P_j\big( t^{\frac{1}{2}} q^{\frac{i}{2}} , t^{\frac{-1}{2}} q^{\frac{-i}{2}} \big| 0, 0\big) \ \dfrac{ P_j\big( t^{\frac{1}{2}} q^{\frac{i}{2}} , t^{\frac{-1}{2}} q^{\frac{-i}{2}} \big| p/s, s \big) }{ P_j( t^{\frac{1}{2}} q^{\frac{i}{2}} , t^{\frac{-1}{2}} q^{\frac{-i}{2}} \big| p/s, 0 \big) } \cdot \dfrac{Z(g,\beta,s,p)}{Z(g,\beta)} = \\
& \nonumber S_{ij} \dfrac{Z(g,\beta,s,p)}{Z(g,\beta)} \left( 1 + w_0(j) \dfrac{P_{j+2} \big( t^{\frac{1}{2}} q^{\frac{i}{2}} , t^{\frac{-1}{2}} q^{\frac{-i}{2}} \big) }{P_{j} \big( t^{\frac{1}{2}} q^{\frac{i}{2}} , t^{\frac{-1}{2}} q^{\frac{-i}{2}} \big)} \frac{p}{s} + w_1(j) \frac{p}{s} + w_2(j) \dfrac{P_{j-2} \big( t^{\frac{1}{2}} q^{\frac{i}{2}} , t^{\frac{-1}{2}} q^{\frac{-i}{2}} \big)}{P_{j} \big( t^{\frac{1}{2}} q^{\frac{i}{2}} , t^{\frac{-1}{2}} q^{\frac{-i}{2}} \big)} \frac{p}{s} + O(p^2) \right) 
\end{align}
\smallskip\\
where we used Proposition 5 and $w_k(j) = c_k(j) - c_k(j)\big|_{s = 0}$. From the definition of $S_{i,j}$, this is the same as \linebreak

\begin{align}
\mbox{rhs} = \dfrac{Z(g,\beta,s,p)}{Z(g,\beta)} \left( S_{i,j} + w_0(j) S_{i,j+2} \frac{p}{s} + w_1(j) S_{i,j} \frac{p}{s} + w_2(j) S_{i,j-2} \frac{p}{s} + O(p^2) \right) 
\label{EquationRHS}
\end{align}
\smallskip\\
Finally, equating l.h.s. given by eq. (\ref{EquationLHS}) with r.h.s. given by eq. (\ref{EquationRHS}), what we need to prove is that there exists constant $\eta$ such that for all $i,j \geq 0$ we have

\begin{align}
& \nonumber u_0(i) t q^{i+1} S_{i+2,j} + u_1(i) S_{i,j} + u_2(i) q^{-i+1} t^{-1} S_{i-2,j} + \\
& \nonumber c_0(j) t q^{j+1} S_{i,j+2} + c_1(j) S_{i,j} + c_2(j) q^{-j+1} t^{-1} S_{i,j-2} + \\
& \nonumber v_0(i) t q^{i+1} S_{i+2,j} + v_1(i) S_{i,j} + v_2(i) q^{-i+1} t^{-1} S_{i-2,j} = \\
& w_0(j)/s S_{i,j+2} + w_1(j)/s S_{i,j} + w_2(j)/s S_{i,j-2} + \eta S_{i,j}
\label{FinalIdentity}
\end{align}
\smallskip\\
We are going to show that $\eta = 2(1-t)(q+t)/(1-q)/t$. To prove this identity, we apply the definition of $S_{i,j}$ and use the well-known product formula for Macdonald polynomials evaluated at the special point, 

\begin{align}
P_i\big( t^{\frac{1}{2}} , t^{\frac{-1}{2}} \big) = t^{-i/2} \dfrac{1-t q^i}{1-t} \prod\limits_{m = 0}^{i-1} \dfrac{ 1-t^2 q^m}{1-t q^{m+1}}
\end{align}
\smallskip\\
which implies

\begin{align}
\dfrac{P_{i+2}\big( t^{\frac{1}{2}} , t^{\frac{-1}{2}} \big)}{P_{i}\big( t^{\frac{1}{2}} , t^{\frac{-1}{2}} \big)} = \dfrac{(1-t^2 q^i) (1-t^2 q^{i+1})}{t (1-t q^i) (1-t q^{i+1})}, \ \ \ \ \ \dfrac{P_{i-2}\big( t^{\frac{1}{2}} , t^{\frac{-1}{2}} \big)}{P_{i}\big( t^{\frac{1}{2}} , t^{\frac{-1}{2}} \big)} = \dfrac{t (1-t q^{i-1}) (1-t q^{i-2})}{ (1-t^2 q^{i-1}) (1-t^2 q^{i-2})}
\end{align}
\smallskip\\
Dividing both sides of eq. (\ref{FinalIdentity}) by $P_i\big( t^{\frac{1}{2}} , t^{\frac{-1}{2}} \big)$, we obtain

\begin{align}
& \nonumber \alpha_0(j) P_{j+2}(z, z^{-1}) + \alpha_1(i) P_{j}(qz, q^{-1}z^{-1}) + \\
& \alpha_2(i) P_{j}(q^{-1}z, qz^{-1}) + \alpha_3(i,j) P_{j}(z, z^{-1}) + \alpha_4(j) P_{j-2}(z, z^{-1}) = 0
\label{FinalIdentity3}
\end{align}
\smallskip\\
where we denoted $z = t^{1/2} q^{i/2}$. The coefficients here are expressed as

\begin{align}
\alpha_0(j) = c_0(j) t q^{j+1} - w_0(j)/s = \dfrac{q^{j+1} t(1-t)}{(1-q)} 
\end{align}
\begin{align}
\alpha_1(i) =  \big( u_0(i) + v_0(i) \big) t q^{i+1} \dfrac{(1-t^2 q^i) (1-t^2 q^{i+1})}{t (1-t q^i) (1-t q^{i+1})} = \dfrac{ - q^{1+i} (1 - t^2 q^{i}) (1 - t) (1 - t^2 q^{i+1})}{(1 - q) (1 - t q^i) (1 - t q^{i+1})}
\end{align}
\begin{align}
\alpha_2(i) =  \big( u_2(i) + v_2(i) \big) q^{-i+1}  t^{-1} \dfrac{t (1-t q^{i-1}) (1-t q^{i-2})}{ (1-t^2 q^{i-1}) (1-t^2 q^{i-2})} = \dfrac{- q^{-i} (1 - t) q (1 - q^{i}) (1 - q^{i-1})}{(1 - t q^{i-1}) (1 - q) (1 - t q^{i})}
\end{align}
\begin{align}
\hspace{-3ex} \alpha_3(i,j) = u_1(i) + v_1(i) + c_1(j) - w_1(j)/s - \eta = \dfrac{(1 - t)^2 (q^i - q^j) (1 - t^2 q^{i+j}) (1 + q) (q -t)}{q(1 - q^{i+1} t) (1 - q) (1 - tq^{i-1}) (1 - t q^{j+1}) (1 - q^{j-1} t)}
\end{align}
\begin{align}
\alpha_4(j) = c_2(j) q^{-j+1}  t^{-1} - w_2(j)/s = \dfrac{q^{1-j} (1 - q^j) (1 - q^{j-1}) (1 - t) (1 - t^2 q^{j-1}) (1 - t^2 q^{j-2})}{t (1 - q) (1 - t q^j)(1 - t q^{j-1})^2 (1 - t q^{j-2})}
\end{align}
\smallskip\\
To prove eq. (\ref{FinalIdentity3}), let us introduce operators

\begin{align}
{\cal O}_A f(z) = \dfrac{1 - t z^2}{1 - z^2} f( q^{\frac{1}{2}} z) + \dfrac{1 - t z^{-2}}{1 - z^{-2}} f( q^{\frac{-1}{2}} z), \ \ \ \ \ {\cal O}_B f(z) = (z + 1/z) f(z)
\label{OAOB}
\end{align}
and

\begin{align}
{\cal O}_C = {\cal O}_A {\cal O}_B - q^{\frac{-1}{2}} {\cal O}_B {\cal O}_A  
\label{OC}
\end{align}
\smallskip\\
We will use the known result,

\begin{align}
{\cal O}_C P_j(z, z^{-1}) = q^{\frac{j-1}{2}} t ( q - 1) P_{j+1}(z, z^{-1}) - \dfrac{q^{-\frac{j+1}{2}} (1-q) (1-q^{j}) (1-t^2 q^{j-1})}{(1 - t q^j) (1 - t q^{j-1})} P_{j-1}(z, z^{-1})
\label{OCPieri}
\end{align}
\smallskip\\
which is a corollary of Pieri formulas for Macdonald polynomials. Then one the one hand, from eq. (\ref{OCPieri}),

\begin{align}
-\dfrac{(1-t) q^{3/2}}{t (1-q)^3 } {\cal O}^2_C P_j(z, z^{-1}) = - \alpha_0(j) P_{j+2}(z,z^{-1}) + \alpha_5(j) P_{j}(z,z^{-1}) - \alpha_4(j) P_{j-2}(z,z^{-1})
\label{Eq1}
\end{align}
\smallskip\\
while on the other hand, from eqs. (\ref{OAOB}) and (\ref{OC}),

\begin{align}
-\dfrac{(1-t) q^{3/2}}{t (1-q)^3 } {\cal O}^2_C P_j(z, z^{-1}) = \alpha_1(i) P_{j}(qz, q^{-1}z^{-1}) + \alpha_6(i) P_{j}(z,z^{-1}) + \alpha_2(i) P_{j}(q^{-1}z, qz^{-1})
\label{Eq2}
\end{align}
\smallskip\\
with

\begin{align}
\alpha_5(j) = \dfrac{(1 - t) (2 q^{2j + 1} t^2 - q^{2 + j} t - q^{1+j} t^2 - q^{2+j} + 2 q^{1+j} t - q^j t^2 - q^{1+j} - q^j t + 2 q)}{q (1 - tq^{j-1}) (1 - t q^{j+1} ) (1 - q)}
\end{align}
\begin{align}
\alpha_6(i) = \dfrac{-(1 - t) (-2 q t z^4 + q^2 t z^2 + q t^2 z^2 + q^2 z^2 - 2 q t z^2 + t^2 z^2 + q z^2 + t z^2 - 2 q t)}{t (1 - q) (q - z^2) (1 - q z^2)}
\end{align}
\smallskip\\
and we remind that $z = t^{1/2} q^{i/2}$. Substracting eq.(\ref{Eq1}) from eq.(\ref{Eq2}) and observing that

\begin{align}
\alpha_3(i,j) = \alpha_6(i) - \alpha_5(j)
\end{align}
\smallskip\\
completes the proof of eq. (\ref{FinalIdentity3}) and therefore, the proof of the theorem.

\end{proof}

\section{Conclusion}

In this paper we suggested an elliptic generalization of CMM integral identities, where trigonometric ($q$-deformed) Vandermonde is replaced by an elliptic version -- a product of theta functions -- and Macdonald polynomials are replaced by Shiraishi functions. This can be regarded as a development in the direction of the paper \cite{AS23} because we continue to demonstrate that Shiraishi functions are appropriate coordinate-elliptic generalizations of Macdonald polynomials and enjoy many of the properties of the original Macdonald polynomials, properly generalized. In \cite{AS23} generalization of the $SL(2,{\mathbb Z})$ modular $S$ and $T$ matrices was considered. Here, instead of finite-dimensional representations, we consider the continuous case where CMM identities play a similar role.

It is interesting to note that, just like in \cite{AS23}, the identities survive elliptic generalization provided that appropriate transformations are applied to the elliptic parameters $(p,s)$, such as $s \rightarrow s/p$ and $p \rightarrow p/s$. These transformations are typically linear transformations of the logarithms of the parameters, therefore they are multiplicative transformations of parameters theirselves. It is expected that this group of transformations can be further enriched to $SL(3,{\mathbb Z})$ transformations \cite{AS23}, however, the precise formulas have not been presented yet.
 
We are hopeful that results of current paper could attract attention to the problem of construction of elliptic refined Chern-Simons theory, and associated elliptic knot/link invariants, existence of which was discussed in \cite{AS23}. It appears that Shiraishi functions do possess many if not all of the required properties for this extension of refined Chern-Simons theory to exist. We hope to return to this subject in later work.
 
\section{Acknowledgments}

We would like to thank A. Morozov, A. Mironov, A. Popolitov, A. Anokhina, S.Arthamonov, and J. Shiraishi for valuable discussions. The research of S. R. Shakirov was supported in part by the State Assignment of the Institute for Information Transmission Problems of the RAS.

\end{document}